\documentclass[12pt, a4paper, reqno]{amsart}

\usepackage{amsthm, mathtools, amssymb, amsfonts, stmaryrd, latexsym, mathrsfs, todonotes, eufrak, color}
\usepackage{graphicx}
\usepackage[english]{babel}
\usepackage[all]{xy}
\usepackage{nicefrac}
\usepackage{mparhack}
\usepackage[normalem]{ulem}
\usepackage{color}
\usepackage{blindtext}
\usepackage[inline]{enumitem}
\usepackage{fancyhdr}
\usepackage{enumitem}
\usepackage{algorithmic}
\usepackage{algorithm}
\usepackage{hyperref}

\theoremstyle{plain}
\newtheorem{theo}{Theorem}[section]
\newtheorem{lemma}[theo]{Lemma}
\newtheorem{prop}[theo]{Proposition}
\newtheorem{corollary}[theo]{Corollary}
\newtheorem{remark}[theo]{Remark}
\newtheorem{example}[]{Example}

\theoremstyle{definition}
\newtheorem{definition}[]{Definition}

\newcommand{\F}{\mathbb{F}}

\newcommand{\Z}{\mathbb{Z}}
\newcommand{\Q}{\mathbb{Q}}
\newcommand{\R}{\mathbb{R}}

\newcommand{\N}{\mathbb{N}}

\renewcommand{\epsilon}{\varepsilon}

\newcommand{\old}[1]{\sout{#1}}
\newcommand{\new}[1]{{\color{red}{#1}}}

\newcounter{enumi_saved}

\makeatletter
\def\imod#1{\allowbreak\mkern10mu({\operator@font mod}\,\,#1)}
\makeatother

\allowdisplaybreaks

  \title{Finite groups of units of\\finite characteristic rings}
\begin{document}

\maketitle
\begin{center}\author{\textsc{Ilaria Del Corso and Roberto Dvornicich}}
\end{center}

\begin{abstract}
  In \cite[Problem 72]{Fuchs60} Fuchs asked the following question:
  which groups can be the group of units of a commutative ring?  In the
  following years, some partial answers have been given to this
  question in particular cases.  The aim of the present paper is to
  address Fuchs' question when $A$ is a {\it finite characteristic
    ring}. The result is a pretty good description of the groups which can occur as group of units in this case, equipped
  with examples showing that there are obstacles to a ``short''
  complete classification. 

  As a byproduct, we are able to classify all possible cardinalities of the group of units  of a finite characteristic ring, so to answer Ditor's question \cite{ditor}.  
\end{abstract}

\section{Introduction}

In \cite[Problem 72]{Fuchs60} Fuchs 
posed the following problem:

\begin{itemize}
\item[]{\sl Characterize the groups which are the groups of all units in a commutative and associative ring with identity.}
\end{itemize}

In the following years, this question and related ones have been considered by many authors.
A first result is due to  Gilmer \cite{Gilmer63}, who classified the {\it finite commutative rings} $A$ with the property that the group of units $A^*$ is cyclic,  classifying as a byproduct also the possible cyclic groups that arise in this case. 

Further steps in the study of Fuchs' Problem came from
the results by  Hallet and Hirsch \cite{HalletHirsch65}, and subsequently by Hirsch and Zassenhauss
\cite{HirschZassenhaus66}, combined with  \cite{Corner63}. 
From their results one can deduce necessary conditions for a  finite group to be the group of units of a reduced and torsion free ring.

Later on, Pearson and Schneider \cite{PearsonSchneider70} combined the result of Gilmer and the result 
of Hallet and Hirsch to describe explicitly all possible {\it finite cyclic groups} that can occur as $A^*$ for a commutative ring $A$. 

In 1971 Ditor \cite{ditor}, in view of the difficulty of dealing with the general Fuchs Problem, posed the following less general question: 
\begin{itemize}
\item[]{\sl Which whole numbers can be the number of units of a ring? }
\end{itemize}

In his paper he answered this question in the easier case of odd integers.

Other instances on this subject were considered also by other authors,
like Eldridge and Fischer \cite{EldridgeFischer67}, who considered
rings with the descending chain condition, or Dol${\rm \check{z}}$an
\cite{Dolzan02}, who classified the finite rings $A$ for which $A^*$
is a finite group whose order is either a prime power or not divisible
by 4. Recently, Chebolu and Lockridge \cite{CheboluLockridge15} have
studied Fuchs' problem in a different setting, and have been able to
classify the {\it indecomposable abelian groups} that are realizable
as the group of units of a ring.
\vskip0.2cm

The aim of the present paper is to  address Fuchs' and Ditor's
questions when $A$ is a {\it finite characteristic ring}. 
We call an abelian  group {\it realizable} if it is the group of units of some finite characteristic ring.

We first show that the general problem can be reduced to the study of the group of units of finite local rings.  In Theorem \ref{struttura} and Corollary \ref{carfinita} we give necessary conditions for an abelian group to be realizable. 
The conditions found in Theorem \ref{struttura} allow to produce infinite families of non-realizable groups (see Proposition \ref{nonreal}). As a direct consequence, we reobtain the classification of cyclic units groups already found by Gilmer in \cite{Gilmer63} (see Corollary \ref{ciclici}). 

On the other hand,  in the paper we also  give positive answers; in fact,  we  exhibit  large classes of groups that are realizable.
For example, in Propositions \ref{p>2} we show that, for $p>2$, all groups of the form $\F_{p^\lambda}^*\times P^\lambda$ are realizable, for all $\lambda\ge1$ and for all abelian $p$-group $P$.  
Proposition \ref{p=2} contains a similar result for $p=2$.

Moreover, in Section \ref{secex}, we give further examples which are not included in  Propositions \ref{p>2}  and \ref{p=2}.
Both the propositions  and the construction of the examples need the case of rings with characteristic a power of 2 to be treated separately. This is mainly due to the particular behaviour of $(\Z/p^n\Z)^*$ for $p=2$.

Our results allow to exhibit 
infinite families of "new" examples of realizable groups (see Remark \ref{new}) and infinite families of non realizable groups (see Proposition \ref{nonreal}), reducing considerably the space between necessary and sufficient conditions
  for a group to be realizable. In view of the examples in Section \ref{secex}, we believe that  a more precise statement of
  necessary and sufficient conditions would be very complicated.
  
Moreover,  we are able to completely answer Ditor's question for finite characteristic rings
(see Corollary \ref{numericarfinita}) and to construct,  for each possible cardinality $n$, a  finite characteristic ring $A$ with $|A^*|=n$.

As a final remark, we point out that our results include all previously known facts for the case under consideration; some results already proved in \cite{Gilmer63} and in  \cite{Dolzan02} have been derived as particular cases in our paper.

In a forthcoming paper \cite{DDchar0} we investigate Fuchs' and Ditor's Problems in the case of characteristic 0 rings and we obtain a complete classification of the groups of units of reduced rings.

\section{Notation and preliminary results}

Let $A$ be ring with 1 of finite characteristic. As usually,
$A^*$ will denote its multiplicative group, which we shall assume
{\it finite and abelian} throughout the paper.

We recall that a ring is called reduced if its nilradical is trivial, namely, if it does not contain non-zero nilpotents.

\begin{remark}
  \label{ZA} {\rm Let $n=char(A)$, then $\Z/n\Z$ is the fundamental subring of $A$. Let
    $B=\Z/n\Z[A^*]$; then $B\subseteq A$,  and $B^*=A^*$.  
    
     We also note that, the rings $A$ and $B$ have the same nilradical; in fact, the nilradical of $B$ is obviously contained in the nilradical of $A$; on the other hand, if $a\in A$ is such that $a^n=0$ for some $n\in\N$, then $1+a\in A^*\subseteq B^*$, so $a\in B$. In particular, $A$ is reduced if and only if $\Z/n\Z[A^*]$ is reduced. 

If $A$ is a domain, then its characteristic is a prime number $p$ and $\Z/p\Z[A^*]$ is a finite field.
 }   
 \end{remark}
In this paper we are interested in determining the possible abelian groups that occur as groups of units of a finite characteristic ring: as a consequence of the previous remark, we
 may restrict our study to  {\it commutative} and {\it finite rings}.

In the following we shall make repeated use of the following proposition:
\begin{prop}
\label{successioneesatta}
{\sl
Let $A$ be a commutative ring, let $\mathfrak{N}$ denote its nilradical and let $\mathfrak{I}$ be an ideal of $A$ contained in $\mathfrak{N}$. Then the sequence
\begin{equation}
\label{success}
1\to 1+\mathfrak{I}\hookrightarrow A^*\stackrel{\phi}{\to}\left(A/\mathfrak{I}\right)^*\to 1,
\end{equation}
where $\phi(x)=x+\mathfrak{I}$, is exact.
}
\end{prop}
\begin{proof} Clearly  
$$ 1\to 1+\mathfrak{I}\to A^*\to{A^*}/{(1+\mathfrak{I})}\to 1$$
is exact. Moreover, the
 map $\Phi : A^*\to (A/\mathfrak{I})^*$  defined by $x\mapsto x+\mathfrak{I}$ is a homomorphism with kernel equal to $1+\mathfrak{I}$. Further, if $x+\mathfrak{I}\in (A/\mathfrak{I})^*$ and $(x+\mathfrak{I})( y+\mathfrak{I})=1+\mathfrak{I}$ then $xy\in1+\mathfrak{I}\subseteq A^*$, so $x\in A^*$ and $\Phi$ is surjective, so $A^*/(1+\mathfrak{I})\cong( A/\mathfrak{I})^*.$
\end{proof}

The case when $\mathfrak{I}=\mathfrak{N}$ is particularly interesting since, as we shall see in the following,  for finite characteristic rings the exact sequence  \eqref{success} always splits  in this case (see Theorem \ref{struttura}). However,  there are cases when the exact sequence does not split.

\section{General results}
\label{finite}
As noted in the previous section, we  we may suppose that $A$ is finite and commutative, so we will assume it in the following. Under this assumption, we have that $A$ is artinian, then by the structure theorem of artinian rings, $A$ is
isomorphic to a direct product of artinian local rings, 
$$A \cong
A_1\times \dots \times A_s,$$ 
and hence 
$$A^* \cong A_1^*\times \dots
\times A_s^*.$$
This reduces the problem to the study of the cardinality and the structure of the  group of units of  a finite  local ring. We recall that the characteristic of such a ring is a prime power.
The following theorem describes the groups of units in this case.\begin{theo}
\label{struttura}
{\sl 
Let $p$ be a prime, and let $(A, {\mathfrak m})$ be a finite local  ring of characteristic a power of $p$, such that $A^*$ is finite.
Then 
\begin{equation}
\label{formula}
A^*\cong\F_{p^\lambda}^*\times H
\end{equation}
where $A/\mathfrak{m}\cong \F_{p^\lambda}$ and $H=1+\mathfrak{m}$ is an abelian $p$-group.
Moreover, the filtration $\{1+\mathfrak{m}^i\}_{i\ge1}$ has the property that all its quotients   are $\F_{p^\lambda}$-vector spaces and $|H|=p^{\lambda k}$ for some $k\ge0$.

Conversely, for each prime $p$, for each $\lambda\ge1$ and for each  $k\ge0$ there exists a  a  local artinian ring with characteristic a power of $p$ such that $A^*\cong\F_{p^\lambda}^*\times H$ with $|H|=p^{\lambda k}$.

Finally, when restricting to reduced local artinian rings, the groups of units that can be obtained are precisely the multiplicative groups of finite fields.}
\end{theo}

\begin{proof}
In our hypothesis, we have that the nilradical $\mathfrak{N}$ coincides with  $\mathfrak{m}$, so the exact sequence in \eqref{success} can be written as
\begin{equation}\label{se}
1\rightarrow 1+\mathfrak{m}\to A^*\to (A/\mathfrak{m})^*\to 1.
\end{equation}
	
Consider the finite chain 
\begin{equation}
\label{filtration}
\mathfrak{m}\supset
  \mathfrak{m}^2\supset \dots \supset \mathfrak{m}^r=\{0\}.
  \end{equation}
  We have that, for each $i=1,\dots r-1$, the
  quotient $\mathfrak{m}^i/\mathfrak{m}^{i+1}$ is a $A/\mathfrak{m}$-vector space, and hence has
  cardinality $p^{\lambda k_i}$ for some positive integer $k_i$.
  Putting $k=\sum\limits_{i=1}^{r-1} k_i$, we obtain that 
  $|\mathfrak{m}|=p^{\lambda k}$ and hence $|1+\mathfrak{m}|=p^{\lambda k}$. 

We can now apply  Schur-Zassenhaus Theorem (see \cite[9.12]{Robinson95}) to the exact sequence in \eqref{se}: it follows that $A^*\cong\F_{p^\lambda}^*\times (1+\mathfrak{m})$ and equation \eqref{formula}  follows with $H=1+\mathfrak{m}$. 

Moreover, the isomorphism $\mathfrak{m}^i/\mathfrak{m}^{i+1}\rightarrow(1+\mathfrak{m}^i)/(1+\mathfrak{m}^{i+1})$ induced by the map $x\mapsto1+x$ ensures that the quotients of the filtrations $\{\mathfrak{m}^i\}_{i\ge1}$ and $\{1+\mathfrak{m}^i\}_{i\ge1}$  are isomorphic, hence  $(1+\mathfrak{m}^i)/(1+\mathfrak{m}^{i+1})$ is a $\F_{p^\lambda}$-vector space for each $i$. 	

\smallskip

Conversely, let $\lambda\ge 1$, $k\ge0$ and let $f(t)\in (\Z/p^{k+1}\Z)[t]$ be a monic polynomial of degree $\lambda$ such that its reduction modulo $p$, $\bar{f}(t)$, is irreducible. Define 
	$$R=\frac{(\Z/p^{k+1}\Z)[t]}{(f(t))};$$
clearly $R$ is a local artinian ring and its maximal ideal $\mathfrak{m}$ is generated by the class of $p$. 
As above, $R^*\cong\F_{p^\lambda}^*\times (1+\mathfrak{m})$ and it is easy to check that $|1+\mathfrak{m}|=|\mathfrak{m}|=p^{\lambda k}.$

We conclude by noticing that the statement about reduced rings is clear, since a reduced local artinian ring is a field.
\end{proof}

%
  We summarize the results for a general ring of finite characteristic  in the following corollary, which also retrieves some known results as particular cases.\begin{corollary}
  \label{carfinita} {\sl Let $A$ be a ring of finite characteristic.
  Then $A^*$
  is a finite the product of groups as in \eqref{formula}, where each prime  $p|{\rm char}(A)$,
    gives at least one factor.
    
    Moreover, the finite abelian groups that occur as group of units of reduced rings of finite characteristic  are exactly the finite products of multiplicative groups of finite fields.    
    
Finally, if $A$ is a domain, then $A^*$ is the multiplicative group of a finite field.
    }
    \end{corollary}
\begin{proof}
As observed at the beginning of this section, if $A$ is a ring of finite characteristic  with a finite abelian group of units, then $A^* \cong A_1^*\times \dots \times
  A_s^*$, where $A_i$ are local artinian finite rings. Then the information on $A^*$ can be derived by applying  Theorem \ref{struttura} to each $A_i$. \end{proof}

Theorem \ref{struttura} allows to completely answer the question raised by Ditor \cite{ditor} in the case of finite characteristic rings. In fact, we have the following 

\begin{corollary}
  \label{numericarfinita} {\sl The possible values for $|A^*|$ when
    ${\rm char}(A)$ is finite are all products of numbers of type
    $(p^\lambda -1 )p^{\lambda k}$, where $p|{\rm char}(A)$,
    $\lambda\ge 1$ and $k\ge 0$. In this product there exists at least a factor  of type $(p^\lambda -1 )p^{\lambda k}$ for each prime divisor of
    ${\rm char}(A)$. 
    }
    \end{corollary}
    
We note that  the examples in the proof of Theorem \ref{struttura} allow the explicit construction of a ring $A$ with $|A^*|=n$, for each possible $n$.    
\smallskip 

For the sake of completeness we include the following corollary, which is not new and can be found for example in \cite{ditor}.    
\begin{corollary}
{\sl
The finite abelian groups of odd order that occur as groups of units of a ring $A$ are the finite products
of groups of type $\F_{2^\lambda}^*$. In particular, if $|A^*|=p$ is prime, then $p$ is a Mersenne prime. 
}
  \end{corollary}
\begin{proof}
It is enough to note that if $|A^*|$ is odd, then necessarily ${\rm char}(A)>0$ and to apply Corollary \ref{numericarfinita}.
\end{proof}

\section{Families of realizable and non-realizable groups}

Theorem \ref{struttura} does not  solve completely the question of classifying the abelian $p$-groups $H$ which occur as $1+\mathfrak{m}$ in a local artinian ring. 
However,  in view of the constraints on the filtration of the group $H$,
 our results show that the group $H$ can not be any $p$-group if $\lambda>1$,  nor any group of order a power of $p^\lambda$; in fact, we have the following
 
 \begin{prop}
 \label{nonreal}
{\sl
Let $p$ be a prime, let $k$ and $\lambda$ be positive integers and let $H$ be an abelian $p$-group of cardinalirty $p^{\lambda k}$.  If there exists a local artinian ring $A$ such that  $A^*\cong \F_{p^\lambda}^*\times H$,  then the exponent of $H$ is at most $p^k$. 
}\end{prop}
 \begin{proof}
 By Theorem \ref{struttura}, a necessary condition in order that $H$  can appear in \eqref{formula} is that 
$H$ has a filtration whose quotients are $\F_{p^\lambda}$-vector spaces, so since $|H|=p^{\lambda k}$,
 the length of its filtration  can be at most $k$, and every  element of $H$ must have order at most $p^k$.
 \end{proof}
 \begin{example}
 \label{exnuovo}
{\rm


 The group $H\cong\Z/p^3\Z\times(\Z/p\Z)^{2\lambda-3}$ does not occur as $1+\mathfrak{m}$ if $A/\mathfrak{m}=\F_{p^\lambda}$, since $|H|=p^{2\lambda}$ and $H$ contains elements of order $p^3$.
}
\end{example} 

{From the previous proposition it is easy to obtain a complete
  classification of finite cyclic groups of units which arise in the
  case of finite characteristic rings.  This result is not new (as we have already noted, it was firstly proved by Gilmer
  \cite{Gilmer63}), however, we include it since, at least
  for the case $p>2$, it is an immediate consequence of  Proposition \ref{nonreal}. 
  Actually, the proof
	  for the case $p=2$ requires some more effort,  which however well highlights the importance of the condition on the filtration in Theorem
  \ref{struttura}.

\begin{corollary}
\label{ciclici}
{\sl
A finite cyclic group is the group of units of a local artinian ring of characteristic a power of a prime $p$ 
if and only if its order is:
\begin{enumerate}[label={\alph*)}]
\item $p^\lambda-1$ for $\lambda\ge1$;
\item $(p-1)p^k$ for $k\ge1$ and $p>2$;
\item $2$; 
\item  4.
 \end{enumerate}
 A finite cyclic group is the group of units of a finite characteristic  ring  
if and only if its order is the product of a set of pairwise coprime integers of the previous list.
}
\end{corollary}
\begin{proof}
By Theorem \ref{struttura}, using the same notation, we have  $A^*\cong\F_{p^\lambda}^*\times H$ and in particular $H$ must be cyclic: Proposition \ref{nonreal} implies $|H|=1$ or $\lambda=1$. It follows that $|A^*|$ can only be either $p^\lambda-1$ with $\lambda\ge1$ or $(p-1)p^k$ for $k\ge1$. 
For $p>2$ all these groups occur, for example with 
$A=\F_{p^\lambda}$ and $A=\Z/p^{k+1}\Z$.

For $p=2$ we obtain all groups of order $2^\lambda-1$ by choosing $A=\F_{2^\lambda}$. We obtain the groups of order 2 and 4 by choosing $A=\Z/4\Z$ or $(\Z/2\Z)[x]/(x^3)$. 

On the other hand, when the characteristic of $A$ is a power of $2$, no other case is possible. In fact, since $(\Z/8\Z)^*$ is not cyclic, necessarily the characteristic of $A$ can only be 2 or 4. We also note that in both cases $2x=0$ for all $x\in \mathfrak{m}$ (if not, $\pm1\pm x$ would be 4 distinct elements with the same square, and this is not possible in a cyclic group). 
Assume that $|A^*|=2^k$; then $\lambda=1$ and $H= 1+ \mathfrak{m}$ is cyclic of order $2^k$. 
Let $1+\mu$ be a generator of $1+\mathfrak{m}$. The only filtration of  $1+\mathfrak{m}$ which satisfies the condition of Theorem \ref{struttura} is 
$$\{(1+\mathfrak{m})^{2^i}\}_{1\le i\le k}.$$
It follows that the length of the filtration in  \eqref{filtration} is $r-1=k$ and $\mu^r=0$.
Now, since $2\mu=0$,  and $1+\mu$ has order $2^k$, we have $(1+\mu)^{2^{k-1}}=1+\mu^{2^{k-1}}\ne1$. Therefore, $2^{k-1}<r=k+1$, so $k=1$ or $k=2$.

Finally, the statement about a general ring $A$ of finite characteristic follows recalling that its group of units $A^*$ is the product of groups of units $A_i^*$ ($i=1,\dots,r$) of local artinian rings, so it is cyclic if and only if all the $A_i^*$'s are cyclic and their orders are pairwise coprime.
\end{proof}

On the positive side, we are able to exhibit a large class of new groups that are realizable,   
in the sense that they can not be in general obtained as a product of realizable cyclic groups (see Remark \ref{new}).
In Proposition \ref{p>2} we show that for each $p>2$, $\lambda\ge1$ and for each abelian $p$-group $P$ one can construct a ring $A$ for which $H=P^\lambda$.  
In particular, for $\lambda=1$ all abelian $p$-groups are possible.

The case $p=2$ is exceptional and needs to be treated separately:  a class of groups (most of them new) which can be realized in this case is given in  Proposition \ref{p=2}. 
 
 In the next section we will exhibit examples of realizable groups which are not covered by Propositions  \ref{p>2} and \ref{p=2}.
 
\begin{prop}
\label{p>2}
{\sl
Let $p>2$ be a prime. For each $\lambda\ge1$ and for each finite abelian $p$-group $P$ there exists a local artinian ring $A$ such that 
$$A^*\cong\F_{p^\lambda}^*\times P^\lambda.$$
}
\end{prop}
\begin{proof}
Consider the decomposition of the group $P$ as a product of cyclic groups  $P\cong \Z/p^{a_0}\Z\times \Z/p^{a_1}\Z\times\cdots\times \Z/p^{a_r}\Z$ with $a_0\ge a_1\ge\cdots\ge a_r$, 
and let $f(t)\in( \Z/p^{a_0+1}\Z)[t]$ be a monic polynomial of degree $\lambda$ such that its reduction modulo $p$, $\bar{f}(t)$, is irreducible. 
As in the proof of Theorem \ref{struttura}, define
$$R=\frac{(\Z/p^{a_0+1}\Z)[t]}{(f(t)) }$$
and 
$$A=\frac{R[x_1,\dots x_r]}{( p^{a_i}x_i,x_ix_j)_{1\le i,j\le r} }.$$
We already observed that $R$ is a local artinian ring with maximal ideal $\mathfrak{m}_R=(p)$. We note that  also $A$ is a local artinian ring with maximal ideal $\mathfrak{m}_A=(p,\bar x_1,\dots,\bar  x_r)$; in fact, 
$${A}/{\mathfrak{m}_A}\cong{R}/{\mathfrak{m}_R}\cong\F_{p^\lambda},$$
and  $\mathfrak{m}_A$ is the unique prime ideal of $A$ since all its elements are nilpotent. 

We will show that $A^*\cong\F_{p^\lambda}^*\times P^\lambda$.

Let $\mathfrak{I}=(\bar x_1,\dots,\bar x_r)\subset \mathfrak{m}_A$.
By Proposition \ref{successioneesatta} we have the exact sequence
\begin{equation*}
1\to 1+\mathfrak{I}\hookrightarrow A^*\stackrel{\phi}{\to}\left(A/\mathfrak{I}\right)^*\to 1.
\end{equation*}
Now,  it is clear that $(A/\mathfrak{I})^*\cong R^*\subset A^*$ so the exact sequence splits and $A^*\cong R^*\times (1+\mathfrak{I})$.

We note that $1+\mathfrak{I}\cong\mathfrak{I}$
since the map $x\mapsto1+x$ is a group isomorphism.
Now, $\mathfrak{I}\cong\left(\Z/p^{a_1}\Z\times\cdots\times \Z/p^{a_r}\Z\right)^\lambda$
since, as an abelian group, it is generated by the elements $\{\bar t^j\bar x_i\}_{0\le j<\lambda, 1\le i \le r}$ and $\bar t^j\bar x_i$ has additive order $p^a_i$, since $a_1\le a_0+1$. 
On the other hand, by Theorem \ref{struttura}, we know that $R^*\cong \F_{p^\lambda}\times (1+\mathfrak{m}_R)$, so 
we are left to prove that 
 $1+\mathfrak{m}_R\cong \left(\Z/p^{a_0}\Z\right)^\lambda$. 
 We note that $R$ is a free $(\Z/p^{a_0+1}\Z)$-module with basis $1,\bar t,\dots,\bar t^{\lambda-1}$, so $R\cong \left(\Z/p^{a_0+1}\Z\right)^\lambda$ as a group, hence $\mathfrak{m}_R=pR\cong \left(\Z/p^{a_0}\Z\right)^\lambda$. Thus, we are reduced to show that
 $1+\mathfrak{m}_R\cong\mathfrak{m}_R$. Since both are $p$-groups, it is enough to prove that the subgroup  $(1+\mathfrak{m}_R)_{p^l}$ of elements of exponent $p^l $ of $1+\mathfrak{m}_R$ has the same cardinality of the subgroup $(\mathfrak{m}_R)_{p^l}$ of elements of exponent $p^l $ of $\mathfrak{m}_R$. This is a direct consequence of the following lemma.
\end{proof}
\begin{lemma}
\label{esponentepl} 
{\sl 
Let $\mu\in\mathfrak{m}_R=pR$ and let $l\ge0$. Then,
$$(1+\mu)^{p^l}=1\iff p^l\mu=0.$$
}
\end{lemma} 
\begin{proof}
By the binomial theorem,
\begin{equation}
\label{binomio}
(1+\mu)^{p^l}=1+p^l\mu+\sum_{k=2}^{p^l}\binom{p^l}{k}\mu^k.
\end{equation}

We claim that, if $\mu\in p^iR$ ($i\ge1$) and hence $p^l\mu\in p^{l+i}R$, then $\binom{p^l}{k}\mu^k\in p^{l+i+1}R$, for all $k\ge2$.
In fact, clearly $\mu^k\in p^{ik}R$; moreover, it is a classical result that if $p^c\mid\mspace{-2mu}\mid k$, then $p^{l-c}\mid\mspace{-2mu}\mid\binom{p^l}{k}$, so, since $p>2$ and $k\ge2$, we have that
$p^{l-k+2}$ divides $\binom{p^l}{k}$ and
$$p^{l+(i-1)k+2}\mid\binom{p^l}{k}\mu^k.$$

Now, $l+(i-1)k+2\ge l+i+1$ since this is equivalent to $(i-1)(k-1)\ge0$  (recall that $i\ge1$ and $k\ge2$), and the claim follows.

In conclusion, if $p^l\mu=0$ (namely, if $l+i\ge a_0+1$),  then $\binom{p^l}{k}\mu^k=0$ also for all $k\ge 2$, so $(1+\mu)^{p^l}=1$; on the other hand,
if $p^l\mu\ne0,$ then there exists $i$ with $l+i\le a_0$ such that $p^l\mu\in p^{l+i}R\setminus p^{l+i+1}R,$ so 
	$$ (1+\mu)^{p^l}\equiv1+p^l\mu\not\equiv1\pmod{p^{l+i+1}}.$$

\end{proof}

The case $\lambda=1$ of Proposition \ref{p>2} immediately gives the following
\begin{corollary}
\label{corp>2}
Let $p>2$ be a prime. 
For each finite abelian $p$-group $P$ there exists a finite local ring $A$ such that 
$$A^*\cong\F_{p}^*\times P.$$
\end{corollary}

\begin{remark}
\label{new}
{\rm
Proposition \ref{p>2} produces infinite examples of "new" groups of units. In fact,
 if $p>2$ is not a Mersenne prime, and $P$ is a non cyclic abelian $p$-group, then
the group
$\F_{p^\lambda}^*\times P^\lambda$
is realizable but is not a product of cyclic realizable groups. 
 In fact, it is enough to observe that $2^a-1=p^b$ has no solution in view of the famous Catalan-Mih${\rm \check{a}}$ilescu Theorem. 

}
\end{remark}

\begin{remark}
{\rm
Proposition \ref{p>2} can not be generalised to $p=2$. In fact, putting $p=2$ and $\lambda=1$ we would obtain that all abelian 2-groups are groups of units but, as recalled in Corollary \ref{ciclici}, this is not true already for cyclic groups.

The peculiarity of the case $p=2$ lies in the fact that the group $(\Z/2^a\Z)^*$ is not cyclic in general. However, in the following proposition we exhibit a large class of groups of units of rings with characteristic a power of 2.
}
\end{remark}
\begin{prop}
\label{p=2}
{\sl 
Let $P$ be a finite abelian 2-group and  let $2^a$ be its exponent.  Then, for each integer $a_0\ge a-1$ and for each $\lambda \ge1$ there exists a  local artinian ring $A$ such that 
$$A^*\cong\F_{2^\lambda}^*\times\Z/2\Z\times \Z/2^{a_0-1}\Z\times(\Z/2^{a_0}\Z)^{\lambda-1}\times P^\lambda.$$
}
\end{prop}
\begin{proof}
Consider the decomposition of $P$ as a product of cyclic groups,  $P\cong  \Z/2^{a_1}\Z\times\cdots\times \Z/2^{a_r}\Z$ with $a_1\ge\cdots\ge a_r$, so $a=a_1$.
Let $f(t)\in \Z[t]$ be a monic polynomial of degree $\lambda$ such that its reduction modulo $2$, $\bar{f}(t)$, is irreducible. 
As in Proposition \ref{p>2}, define
$$R=\frac{(\Z/2^{a_0+1}\Z)[t]}{(f(t)) }$$
and 
$$A=\frac{R[x_1,\dots x_r]}{( 2^{a_i}x_i,x_ix_j)_{1\le i,j\le r} };$$
$R$ and $A$ are local artinian rings with maximal ideals $\mathfrak{m}_R=2R$ and $\mathfrak{m}_A=(2, \bar x_1, \dots,\bar x_r)$, respectively.
We will show that $A$ has the required group of units.
Again, letting $\mathfrak{I}=(\bar x_1,\dots,\bar x_r)\subset \mathfrak{m}_A$, we can argue as in the proof of Proposition \ref{p>2} and get 
 $A^*\cong R^*\times (1+\mathfrak{I})$ and $1+\mathfrak{I}\cong\mathfrak{I}\cong P^\lambda.$

Moreover, $R^*\cong \F_{2^\lambda}^*\times (1+2R)$. Then we have to study the group $1+2R$: as in Proposition \ref{p>2} we have $R\cong (\Z/2^{a_0+1}\Z)^\lambda$ as a group, so $2R\cong (\Z/2^{a_0}\Z)^\lambda$, but, contrary to the case $p>2$, $1+2R$ is not always isomorphic to $2R$. 

When $a_0=1$, it is easy to see that the group $1+2R$ is elementary abelian of cardinality $2^\lambda.$

Assume now $a_0\ge2$.
To determine the structure of $1+2R$ it is enough to know the structure of its subgroup $(1+2R)^2$ and of the quotient $(1+2R)/(1+2R)^2$, since this quotient has the same number of generators as the group $1+2R$.

We begin by studying the quotient. Clearly, it is an elementary abelian 2-group, so it is enough to determine its order; now, $|(1+2R)/(1+2R)^2|=|\ker\varphi|$ where $\varphi: (1+2R)\to (1+2R)^2$ is defined by $u\mapsto u^2.$ 
Let $1+2\overline{p(t)}\in 1+2R$, where $\overline{p(t)}\in R$ and $\deg p(t)<\lambda$. 
We have
$$
(1+2\overline{p(t)})^2=1\iff 4\overline{p(t)}(1+\overline{p(t)})=0\iff \overline{p(t)}(1+\overline{p(t)})\in2^{a_0-1}R.
$$
Since $R$ is a local ring,  one between $\overline{p(t)}$ and $1+\overline{p(t)}$ is invertible, so  either $\overline{p(t)}\in2^{a_0-1}R$ or $1+\overline{p(t)}\in2^{a_0-1}R$ and actually exactly  one of these two relations holds. 
Hence, $\ker\varphi=2^{a_0}R\cup 1+2^{a_0}R$ and 
$$|(1+2R)/(1+2R)^2|=|\ker\varphi|=2|2^{a_0}R|=2^{\lambda+1}.$$

Next, to study $(1+2R)^2$ we find it convenient to study first the larger group $1+4R$. We observe that $|(1+2R)/(1+4R)|=2^\lambda$ (recall that $a_0\ge2$), hence $|(1+4R)/(1+2R)^2|=2$.
As in Proposition \ref{p>2},  we prove that $1+4R$ is isomorphic to $4R\cong(\Z/2^{a_0-1}\Z)^\lambda$ by showing that the two groups have the same number of elements of exponent $2^l$ for all $l$. This is a consequence of the following slight variation of Lemma \ref{esponentepl}.
\begin{lemma}
\label{binomio2}
{\sl
Let $\mu\in4R$ and let $l\ge0$. Then,
$$(1+\mu)^{2^l}=1\iff 2^l\mu=0.$$
}
\end{lemma}
\begin{proof}
Let $\mu\in2^iR$ with $i\ge2$. Consider the expansion of $(1+\mu)^{2^l}$.  Clearly $2^{l+i}\Vert 2^l\mu$;  as in the proof of Lemma \ref{esponentepl} it is enough to show that $2^{l+i+1}$ divides $\binom{2^l}{k}\mu^k$ for $k\ge2$. 
In this case, $2^{l-k+1}$ divides  $\binom{2^l}{k}$, so 
$$ 2^{l-k+1+ik}\mid\binom{2^l}{k}\mu^k$$
Now, $l-k+1+ik\ge l+i+1$ since this is equivalent to $i(k-1)\ge k$ and $i,k\ge2.$
\end{proof}
Putting everything together we get
$$(1+2R)^2\cong  \Z/2^{a_0-2}\Z\times(\Z/2^{a_0-1}\Z)^{\lambda-1}$$
and 
$$(1+2R)\cong  \Z/2\Z\times\Z/2^{a_0-1}\Z\times(\Z/2^{a_0}\Z)^{\lambda-1}.$$
\end{proof}

\section{Further examples}
\label{secex}

In this section we give examples showing that the groups $H$ obtained in Propositions \ref{p>2} and \ref{p=2}  do not exhaust all realizable groups.

\begin{example}
\label{esempio-p}
{\rm Let $p>2$, $f(t)$ be a monic polynomial of degree 2 which is irreducible modulo $p$ and let
$$A=\frac{(\Z/p^2\Z)[t,x]}{(f(t),px^2,x^{p-1}+p)}.$$
We claim that $A$ is a finite local artinian ring and, denoting by $\mathfrak{m}$ its maximal ideal, we have $A^*\cong\F_{p^2}\times (1+\mathfrak{m})$ and $1+\mathfrak{m}$ is not the square of a group, so this group of units does not belong to the set of  groups described in Proposition \ref{p>2}.
Defining
$$R=\frac{(\Z/p^2\Z)[t]}{(f(t))},\quad B=\frac{R[x]}{(x^{p-1}+p)}$$
we have
$$A=\frac{B}{(p\bar x^2)}.$$

Denoting by $\bar x$ the class of $x$ in $B$, we have 
$$B=\oplus_{i=0}^{p-2}R\bar x^i$$ as an $R$-module. 
It follows that, as an $R$-module, the ideal $(p\bar x^2)$ is equal to $\oplus_{i=2}^{p-2}Rp\bar x^i$, hence
	$$A=\frac{\oplus_{i=0}^{p-2}R\bar x^i}{\oplus_{i=2}^{p-2}Rp\bar x^i}
	\cong R^2\oplus\left(\frac{R}{pR}\right)^{p-3}$$
as an $R$-module, and $|A|=p^{2p+2}$. Clearly $A$ is a finite local artinian ring, since $B$ is such; its maximal ideal is  
 $\mathfrak{m}=(p,\bar{\bar x})=(\bar {\bar x})$, where $\bar{\bar x}$ denotes the class of ${\bar x}$ in $A$. 
This gives the $R$-module isomrphism  
 $$\mathfrak{m}=pR\oplus R\bar {\bar x}\oplus_{i=2}^{p-2} R\bar{\bar x}^i\cong R\oplus\left(\frac{R}{pR}\right)^{p-2},$$
so,  $|\mathfrak{m}|=p^{2p}$ and $A/\mathfrak{m}\cong\F_{p^2}$, {\it i.e.}, in the notation of Proposition \ref{p>2}, $\lambda=2$.

We show that 
$$1+\mathfrak{m} \cong \Z/p^2\Z\times(\Z/p\Z)^{2p-2}$$ 
and hence it is not the square of a group. 
In fact, it is enough to prove that in $ 1+\mathfrak{m}$ there are exactly $p^{2p-1}$ elements of exponent $p$.

Let $a\in 1+\mathfrak{m}$, then  $a$ can be written as $a= 1+ a_0  p+ a_1\bar {\bar x}+\dots+a_{p-2}\bar {\bar x}^{p-2}$ 
 where $a_1\in R$ is  uniquely determined, and $a_i\in R$ is uniquely determined  modulo $p$ if $i\ne1$.
Then $a^p=1+p( a_1- a_1^p)\bar {\bar x}=1$ if and only if the reduction of $a_1$ modulo $p$ belongs to $\F_p$. The elements $a$ verifying this condition are $p^{2p-1}$ since there are $p^2$ possibilities for $ a_i$ if $i\ne1$, and $p^3$ possibilities for $a_1$, and the claim follows.
}
\end{example}
\begin{example}
\label{esempio-p=2}
{\rm
For $p=2$  and $a_0\ge3$, consider the ring 
$$A=\frac{(\Z/2^{a_0+1}\Z)[t,x]}{(t^2+t+1,4x,x^2+2x)}.$$
We show that $A$ is a finite local artinian ring, but $A^*$ does not belong to the set of groups described in Proposition \ref{p=2}. Theorem \ref{struttura}
 guarantees that $A^*=(A/\mathfrak{m})^*\times (1+\mathfrak{m})$, where $\mathfrak{m}$ is the maximal ideal of $A$.
 
 Similarly to above, we can consider the ring 
 $$R=(\Z/2^{a_0+1}\Z)[t]/(t^2+t+1);$$
 we have $\mathfrak{m}_R=\mathfrak{m}\cap R=2R$ and $A/\mathfrak{m}=R\mathfrak{m}_R=\F_4$, so in the previous notation $\lambda=2$.
As seen in the proof of Proposition \ref{p=2}, we have that 
$$1+\mathfrak{m}_R=(1+\mathfrak{m})\cap R^*\cong\Z/2\Z\times \Z/2^{a_0-1}\Z\times\Z/2^{a_0}\Z$$  is a subgroup of $1+\mathfrak{m}$. 

Denote by $\bar x$ the class of $x$ in $A$; arguing as in Example \ref{esempio-p}, we have 
$$\mathfrak{m}=(2,\bar x)=2R\oplus\bar xR\cong R/2^{a_0}R\oplus R/4R$$ as an $R$-module, and $|1+\mathfrak{m}|=|\mathfrak{m}|=2^{2a_0+4}$.

To determine the decomposition of  $1+\mathfrak{m}$ as a product of cyclic groups, we begin by counting its elements of exponent 2. Let $\mu\in \mathfrak{m}$, then $\mu=2\rho+\sigma\bar x$, where $\rho,\sigma\in R$ and are uniquely determined modulo $2^{a_0}R$ and  modulo $4R$, respectively.
Taking into account that $4\bar x=0$, we have
\begin{multline}
(1+\mu)^2=1\iff 4(\rho+\rho^2)+2(\sigma+\sigma^2)\bar x=0\\ 
\iff \rho+\rho^2\in2^{a_0-1}R\ {\rm and}\ \sigma+\sigma^2\in 2R.
\end{multline}
Now, $\rho+\rho^2\in2^{a_0-1}R$ if and only if $\rho\in  2^{a_0-1}R\cup (-1+2^{a_0-1}R)$ and each of the two classes has 4 representatives in $R/2^{a_0}R$, so there are 8 possibilities for $\rho.$ The condition $\sigma+\sigma^2\in 2R$ can be rewritten as $\pi(\sigma)+\pi(\sigma)^2=0$, where $\pi\colon R/4R\to R/2R\cong\F_4$ is the canonical projection. 
This means that $\pi(\sigma)\in\F_2$, so there are $8$ possibilities for $\sigma$. In conclusion, there are $2^6$ elements of exponent 2 in  $1+\mathfrak{m}$. This yields that  the decomposition of $1+\mathfrak{m}$ has 6 cyclic factors. 
Taking into account that $1+\mathfrak{m}$ has order $2^{2a_0+4}$, 6 cyclic factors and a subgroup isomorphic to 
$\Z/2\Z\times \Z/2^{a_0-1}\Z\times\Z/2^{a_0}\Z$, it is immediate to see that
$$
1+\mathfrak{m} \cong  (\Z/2\Z)^3\times\Z/4\Z\times \Z/2^{a_0-1}\Z\times\Z/2^{a_0}\Z
$$
so, for $a_0\ge3$, the group $A^*$ does not belong to the set of groups described in Proposition \ref{p=2}. 
}
\end{example}

\bibliographystyle{amsalpha}
\bibliography{biblio}

\def\Dbar{\leavevmode\lower.6ex\hbox to 0pt{\hskip-.23ex \accent"16\hss}D}
  \def\cfac#1{\ifmmode\setbox7\hbox{$\accent"5E#1$}\else
  \setbox7\hbox{\accent"5E#1}\penalty 10000\relax\fi\raise 1\ht7
  \hbox{\lower1.15ex\hbox to 1\wd7{\hss\accent"13\hss}}\penalty 10000
  \hskip-1\wd7\penalty 10000\box7}
  \def\cftil#1{\ifmmode\setbox7\hbox{$\accent"5E#1$}\else
  \setbox7\hbox{\accent"5E#1}\penalty 10000\relax\fi\raise 1\ht7
  \hbox{\lower1.15ex\hbox to 1\wd7{\hss\accent"7E\hss}}\penalty 10000
  \hskip-1\wd7\penalty 10000\box7}
\providecommand{\bysame}{\leavevmode\hbox to3em{\hrulefill}\thinspace}
\providecommand{\MR}{\relax\ifhmode\unskip\space\fi MR }
\providecommand{\MRhref}[2]{%
  \href{http://www.ams.org/mathscinet-getitem?mr=#1}{#2}
}
\providecommand{\href}[2]{#2}
\begin{thebibliography}{DCD17}

\bibitem[CL15]{CheboluLockridge15}
Sunil~K. Chebolu and Keir Lockridge, \emph{Fuchs' problem for indecomposable
  abelian groups}, J. Algebra \textbf{438} (2015), 325--336.

\bibitem[Cor63]{Corner63}
A.~L.~S. Corner, \emph{Every countable reduced torsion-free ring is an
  endomorphism ring}, Proc. Lond. Math. Soc. \textbf{13} (1963), no.~3,
  687--710.

\bibitem[DCD17]{DDchar0}
Ilaria Del~Corso and Roberto Dvornicich, \emph{On {F}uchs' {P}roblem about the
  group of units of a ring}, in preparation (2017).

\bibitem[Dit71]{ditor}
S.~Z. Ditor, \emph{On the group of units of a ring}, The American Mathematical
  Monthly \textbf{78} (1971), no.~5, 522--523.

\bibitem[Dol02]{Dolzan02}
D.~Dol$\check{\rm z}$an, \emph{Groups of units in a finite ring}, J. Pure Appl.
  Algebra \textbf{170} (2002), no.~2-3, 174--183.

\bibitem[EF67]{EldridgeFischer67}
K.~E. Eldridge and I.~Fisher, \emph{Dcc rings with a cyclic group of units},
  Duke Math. J. \textbf{34} (1967), 243--248.

\bibitem[Fuc60]{Fuchs60}
L.~Fuchs, \emph{Abelian groups}, 3rd ed., Pergamon, Oxford, 1960.

\bibitem[Gil63]{Gilmer63}
R.~W. Gilmer, \emph{Finite rings with a cyclic group of units}, Amer. J. Math.
  \textbf{85} (1963), 447--452.

\bibitem[HH65]{HalletHirsch65}
J.~T. Hallet and K.~A. Hirsch, \emph{Torsion-free groups having finite
  automorphism groups}, J. of Algebra \textbf{2} (1965), 287--298.

\bibitem[HZ66]{HirschZassenhaus66}
K.~A. Hirsch and H.~Zassenhaus, \emph{Finite automorphism groups of torsion
  free groups}, J. London Math. Soc. \textbf{41} (1966), 545--549.

\bibitem[PS70]{PearsonSchneider70}
K.~R. Pearson and J.~E. Schneider, \emph{Rings with a cyclic group of units},
  J. of Algebra \textbf{16} (1970), 243--251.

\bibitem[Rob95]{Robinson95}
D.~J.~S. Robinson, \emph{A course in the theory of groups}, 2nd ed., Graduate
  Texts in Mathematics, vol.~80, Springer-Verlag, New York, 1995.

\end{thebibliography}

\end{document}